\documentclass[a4paper,11pt]{amsart}

\newtheorem{theorem}{Theorem}[section]
\newtheorem{lemma}[theorem]{Lemma}

\theoremstyle{definition}

\theoremstyle{remark}
\newtheorem{remark}[theorem]{Remark}

\numberwithin{equation}{section}

\newcommand{\NN}{\mathbb{N}}
\newcommand{\RR}{\mathbb{R}}
\newcommand{\ZZ}{\mathbb{Z}}

\newcommand{\OO}{\mathcal{O}}
\newcommand{\pa}{\partial}

\newcommand{\om}{\omega}

\newcommand{\AAA}{\mathcal{A}}

\newcommand{\teta}{\tilde{\eta}}

\newcommand{\Hnu}{H_n^{(1)}}
\newcommand{\Hnd}{H_n^{(2)}}

\newcommand{\Ccal}{\mathcal{C}}
\newcommand{\Dcal}{\mathcal{D}}

\newcommand{\ee}{\mathbf{e}}

\DeclareMathOperator{\RE}{\mathrm{Re}}

\begin{document}

\title[Convergence results for a constrained optimization problem]{Helmholtz equation in unbounded domains: some convergence results for a constrained optimization problem}

\author{Giulio Ciraolo}
\address{Dipartimento di Matematica e Informatica, Universit\`a di Palermo, Via Archirafi 34, 90123 Palermo, Italy}
\email{giulio.ciraolo@unipa.it}
\thanks{The paper was completed while the author was visiting \lq\lq The Institute for Computational Engineering and Sciences\rq\rq (ICES) of The University of Texas at Austin, and he wishes to thank the Institute for hospitality and support. The author has been also supported by the NSF-DMS Grant 1361122, the \lq\lq Gruppo Nazionale per l'Analisi Matematica, la Probabilit\`a e le loro Applicazioni\rq\rq (GNAMPA) of the Istituto Nazionale di Alta Matematica (INdAM) and the Firb project 2013 ``Geometrical and Qualitative aspects of PDE''}

\date{}


\begin{abstract}
We consider a constrained optimization problem arising from the study of the Helmholtz equation in unbounded domains. The optimization problem provides an approximation of the solution in a bounded computational domain. In this paper we prove some estimates on the rate of convergence to the exact solution.
\end{abstract}

\maketitle

\section{Introduction}
In this paper, we consider a constrained optimization problem which arises from the computational study of wave propagation in unbounded domains. We are interested in a classical scattering problem, which can be stated as follows. Let $D \subset \RR^d$, $d\geq 2$, be a bounded domain and let $u$ be the solution of
\begin{equation}\label{pb helm intro 1}
\begin{cases}
\Delta u + k^2 u =0, & \textmd{in } \RR^d \setminus \overline{D},\\
u=f, & \textmd{on } \pa D, \\
\lim\limits_{r\to +\infty} r^{\frac{d-1}{2}} \big(\frac{\pa u}{\pa r} - iku \big) =0. &
\end{cases}
\end{equation}
It is well-known that the solution of problem \eqref{pb helm intro 1} can be written explicitly in terms of layer potentials (see \cite{AK} for instance). A challenging problem in real applications is how to approximate the solution of \eqref{pb helm intro 1} in a bounded computational domain $\Omega$, with $D\subset \Omega$. Usually, the goal is to prescribe \emph{transparent boundary conditions} on $\pa \Omega$ in such a way that the corresponding solution approximates the exact solution on a good fashion.

Many methods have been studied and the research on this topic is still very active (see for instance \cite{Be,Ber,CGS,EM,Gi,GrK,Har,Ih,KG,MTH,SS} and references therein).

In a recent paper \cite{CGS}, the authors studied a new approach to the problem of transparent boundary conditions which is based on the minimization of an integral functional arising from the radiation condition at infinity.  The approach in \cite{CGS} works under quite general assumptions on the index of refraction. Indeed, it applies to the study of the Helmholtz equation
\begin{equation}\label{helm intro n variable}
\Delta u + k^2 n(x)^2 u = 0 \quad \textmd{in } \RR^d \setminus \overline{D},
\end{equation}
where the index of refraction $n$ may have some angular dependency at infinity, i.e. $n(x) \to n_{\infty}(x/|x|)$ as $|x| \to +\infty$, as well as some (unbounded) perturbations. The novelty of the method is that it is not based on the knowledge of the exact solution in some exterior domain, but it relies on a different formulations of the radiation condition at infinity (see \cite{PV}); indeed, under suitable assumptions on $n$, the uniqueness of the solution for  \eqref{helm intro n variable} is guaranteed by the radiation condition
\begin{equation}\label{rad cond PV general}
\int_{\RR^d\setminus D} \Big| \nabla u - iknu \frac{x}{|x|} \Big|^2 \frac{dx}{1+|x|} < \infty.
\end{equation}
When a computational domain $\Omega$ is considered, one can try to approximate the solution of \eqref{helm intro n variable}-\eqref{rad cond PV general} by the minimizer $u_\Omega$ of the following constrained optimization problem
\begin{multline}\label{Pb optimiz general}
\textmd{min } \ J_\Omega (v) = \int_{\Omega \setminus D} \Big| \nabla v - iknv \frac{x}{|x|} \Big|^2 \frac{dx}{1+|x|} ,\\
\textmd{where } \ \Delta v + k^2 n(x)^2 v = 0 \quad \textmd{in } \Omega \setminus \overline{D}, \ v=f \textmd{ on } \pa D.
\end{multline}
In \cite{CGS} it was proven that, if $\Omega = B_R$ (a ball of radius $R$ centered at the origin and containing the scatterer), then the minimizer $u_{B_R}$ of \eqref{Pb optimiz general} converges in $H^1_{loc}$ norm to the solution of \eqref{helm intro n variable}-\eqref{rad cond PV general} as $R \to +\infty$.

As already mentioned, this approach works under very general assumptions on $n$ which are not covered by classical methods available in literature (at list in a standard way). Other advantages of this method are: (i) it works for very general choices of $n$ and $\Omega$ \cite{CGS}; (ii) it is of easy implementation since it consists in minimizing a quadratic functional subject to a linear constrain; (iii) it is suitable to be generalized to more general settings, such as the waveguide's case by using the results in \cite{Ci1}--\cite{CM2},\cite{MS}.

If one considers the problem in its full generality, a rigorous study of the rate of convergence properties of this problem appears to be difficult. In the present paper, we shall study the rate of convergence of this approach in the simplest case possible: $n\equiv 1$, $d=2$, $D=B_{R_0}$ and $\Omega=B_R$, with $R_0<R$. It results that the rate of the $H^1_{loc}$ norm convergence to the exact solution is $R^{-1}$ as $R\to +\infty$. Compared to the existing methods in literature, for $n\equiv 1$ this approach gives a slower rate of convergence. However, we believe that the understanding of this simple case gives a hint on the rate of convergence for much more general indexes of refraction, for which the method is more suitable (see also the numerical studies in \cite{CGS2}).

The paper is organized as follows. In Section \ref{section 2} we state the problem, recall the main results in \cite{CGS} and prove some preliminary results. In Section \ref{section 3}, we find an explicit representation of the solution by means of Fourier series. As a consequence, we obtain the convergence estimates.

\section{Preliminaries} \label{section 2}
In this section we introduce some notation and recall some results from \cite{CGS} which will be useful in the rest of the paper. Some preliminary results will also be proven.

Let $R_0>0$ be fixed and let $\psi$ be the solution of
\begin{equation}\label{helm f rad cond}
\begin{cases}
\Delta \psi + k^2 \psi = 0  & \textmd{in } \RR^2 \setminus \bar{B}_{R_0},\\
\psi=f  & \textmd{on } \pa B_{R_0},\\
\displaystyle\lim_{r\to+\infty} r^{\frac{1}{2}} \left( \frac{\pa \psi}{\pa r} - i k \psi \right) =0, &
\end{cases}
\end{equation}
where $r:=|x|$, $f \in L^2(\pa B_{R_0})$. We consider the polar coordinates $x = r (\cos \omega,\sin \omega)$, where $r=|x|$ and $\omega\in [0,2\pi)$, so that
\begin{equation*}
\Delta u = \frac{1}{r} \frac{\pa}{\pa r} \left(r \frac{\pa u}{\pa r} \right) + \frac{1}{r^2} \frac{\pa^2 u}{\pa \om^2},
\end{equation*}
and
\begin{equation}\label{grad polar coord}
\nabla u = \frac{\pa u}{\pa r} \ee_r + \frac{1}{r}  \frac{\pa u}{\pa \omega} \ee_\omega,
\end{equation}
where $\ee_r = x/|x|$ and $\ee_\omega=(-x_2,x_1)/|x|$.

By separating the variables, a solution $u$ of $\Delta u +k^2u=0$ can be written as
\begin{equation}\label{u separated}
u(r,\omega) = \sum_{z \in\ZZ} [a_n J_n(kr)+b_n Y_n(kr)] e^{in\om};
\end{equation}
here $J_n(r)$ and $Y_n(r)$ are, respectively, the Bessel and Neumann functions of order $n$ and they satisfy (see \cite{AS})
\begin{equation} \label{eq Bessel k1}
\frac{1}{r} \left(r \phi' (r) \right)' + \Big( 1 - \frac{n^2}{r^2} \Big) \phi (r) = 0, \quad r \in (0,+\infty).
\end{equation}
We also recall that the Hankel functions of order $n$ and type $1$ and $2$ are given by
\begin{equation}\label{hankel}
\Hnu(r)=J_n(r)+iY_n(r),\quad \Hnd(r)=J_n(r)-iY_n(r).
\end{equation}
By using these notations, the solution $\psi$ to problem \eqref{helm f rad cond} is given by
\begin{equation}\label{psi spher harm}
\psi(x)= \sum_{n\in \NN} f_n \frac{\Hnu(k|x|)}{\Hnu(kR_0)} e^{in \om},
\end{equation}
where
\begin{equation*}
f_n= \frac{1}{2\pi}\int_0^{2\pi} f(\omega)e^{-in\om} d\omega.
\end{equation*}
Here, we used the fact that $\Hnu$ satisfies
\begin{equation*}
    \lim_{r\to \infty} r^{\frac{1}{2}}(\Hnu{}'(r)-i\Hnu(r))=0,
\end{equation*}
which implies that the outgoing solution of \eqref{helm f rad cond} can be written in terms of $\Hnu$, $n\in\NN$.

In \cite{CGS} the authors proposed a method for approximating $\psi$ on
\begin{equation} \label{Ar}
\AAA_R:= B_R \setminus \bar{B}_{R_0},
\end{equation}
which is based on the following minimization problem:
\begin{multline}\label{Pb 1}
\textmd{Minimize } J_{R}(u):= \int_{\AAA_R} \Big{|} \nabla u - iku \frac{x}{|x|} \Big{|}^2 dx, \\ \textmd{where } \Delta u + k^2 u = 0 \textmd{ in } \AAA_R, \ u=f \textmd{ on } \pa B_{R_0}.
\end{multline}
We will denote by $u_{\AAA_R}$ the minimizer of \eqref{Pb 1} (see \cite{CGS} for the existence and uniqueness of the minimizer).
As already mentioned in the introduction, the problem considered in \cite{CGS} is much more general than problem \eqref{helm f rad cond} both for the choice of the domain and for the coefficient $n$, which here is fixed to be $n\equiv 1$ while in \cite{CGS} may have angular dependence as well as perturbations.

The reader will notice that the functional in \eqref{Pb 1} differs from the one mentioned in the Introduction in the absence of the weight $(1+|x|)^{-1}$. However, the two integral formulations of the radiation condition are equivalent when $n\equiv 1$, as follows from an asymptotic expansion at infinity of the solution (see also Section 3 in \cite{CGS}). The choice of the functional without the weight is just to simplify the computations. In the present paper we will deal only with a constant index of refraction, since in this case we know the explicit solution and accurate convergence results can be obtained analytically.

The main results in \cite{CGS} were: (i) the existence and uniqueness of the minimizer $u_{\AAA_R}$ for \eqref{Pb 1}; (ii) $u_{\AAA_R} \to \psi$ in $H^1_{loc}$ norm as $R \to +\infty$. We summarize these results in the following theorem (the results are stated for the particular case studied in this paper).


\begin{theorem}
Let $\psi$ be given by \eqref{psi spher harm}. We have the following results
\begin{itemize}
\item[(i)] for any $R>R_0$ there exists a unique minimizer $u_{\AAA_R}$ of Problem \eqref{Pb 1};
\item[(ii)] $u_{\AAA_R}$ is a solution of
\begin{equation}\label{Helm f}
\Delta u + k^2 u =0 \quad \textmd{in } \AAA_R,\quad u=f \textmd{ on } \pa B_{R_0};
\end{equation}
\item[(iii)] the minimizer of \eqref{Pb 1} converges to $\psi$ as $R\to + \infty$ in the $H_{loc}^1$ norm, that is: for any fixed $\rho> R_0$, we have that
\begin{equation*}
\lim_{R\to +\infty} \| u_{\AAA_R} - \psi \|_{H^1(\AAA_\rho)} = 0.
\end{equation*}
\end{itemize}
\end{theorem}


For any $u,v \in H^1(\AAA_R)$, it will be useful to define the following semidefinite positive hermitian product:
\begin{equation}\label{herm prod I}
\langle u,v \rangle_{R} = \RE \int_{\AAA_R} \nabla u \cdot \nabla \bar{v} dx,
\end{equation}
and the associated seminorm
\begin{equation}\label{seminorm def}
[u]_R= \langle u,v \rangle_{R}^{\frac{1}{2}}.
\end{equation}

We have the following Lemma.

\begin{lemma} \label{lemma herm prod}
Let $u,v \in H^1(\AAA_R)$ and let
\begin{equation*}
u(r,\om)=\sum_{n\in \ZZ} u_n(r)  e^{in\om},\quad v(r,\om)=\sum_{n\in \ZZ} v_n(r) e^{in\om}.
\end{equation*}
Then we have that
\begin{equation}\label{herm prod fourier}
\langle u,v \rangle_{R} = \sum_{n\in \ZZ} \int_{R_0}^R [\rho u_n'(\rho)\bar{v}_n'(\rho) + \frac{n^2}{\rho} u_n(\rho)\bar{v}_n(\rho)]  d\rho.
\end{equation}
\end{lemma}

\begin{proof}
Let $u,v \in C^1(\AAA_R)$. From \eqref{grad polar coord}, we obtain that
\begin{equation*}
\nabla u(r,\omega)=\sum_{n\in \ZZ} [u_n'(r) \ee_r + in \frac{u_n(r)}{r} \ee_\om]  e^{in\om},
\end{equation*}
and an analogous formula holds for $v$. Fubini-Tonelli's Theorem and Parseval's identity yield \eqref{herm prod fourier}.
If $u,v \in H^1(\AAA_R)$, then the conclusion follows from a standard approximation argument.
\end{proof}

\section{Convergence estimates} \label{section 3}
In this section we prove our main result on the convergence of the approximating solution. Our strategy is to write a minimization problem for solutions of the homogeneous Helmholtz equation which is equivalent to Problem \eqref{Pb 1} and then we use Fourier representation to obtain an explicit expression of the minimizer.

For any function $u \in H^1(\AAA_R)$, we define $U\in H^1(\AAA_R)$ as follows:
\begin{equation}\label{U def}
U(x) = e^{-ik|x|} u(x).
\end{equation}
By using this notation, the functional $J_{R}$ in \eqref{Pb 1} can be written as
\begin{equation*}
J_{R}(u)= \langle U, U \rangle_R = [U]_R^2.
\end{equation*}
In the following lemma we write a minimization problem for solutions of the homogeneous Helmholtz equation which is equivalent to Problem \eqref{Pb 1}.

\begin{lemma} \label{lemma v minimizer}
For a fixed $R>R_0$, let $u_{\AAA_R}$ be the minimizer of Problem \eqref{Pb 1} and set
\begin{equation}\label{v minimiz def}
v_{\AAA_R}=u_{\AAA_R}-\psi,
\end{equation}
with $\psi$ given by \eqref{psi spher harm}. Then, $v_{\AAA_R}$ is the unique minimizer of the following problem
\begin{equation}\label{Pb 2}
\textmd{Minimize } I_{R}(v):=\langle V+2\Psi,V \rangle_R,
\end{equation}
where $v$ is a solution of
\begin{equation} \label{Helm 0}
\begin{cases}
\Delta v+k^2v=0, & \textmd{in } \AAA_R; \\
v=0, & \textmd{on } \pa B_{R_0};
\end{cases}
\end{equation}
here $V$ and $\Psi$ are the functions associated to $v$ and $\psi$ by \eqref{U def}, respectively.
\end{lemma}

\begin{proof}
For any $u\in H^1(\AAA_R)$, we define $v\in H^1(\AAA_R)$ by $v=u-\psi$. Hence the functional $J_{R}$ in \eqref{Pb 1} is given by
\begin{equation*}
J_{R}(v+\psi)= \langle \Psi,\Psi \rangle_R + \langle V+2\Psi,V \rangle_R,
\end{equation*}
where $v$ is a solution of \eqref{Helm 0}. Since $\langle \Psi,\Psi \rangle_R$ is fixed, we conclude.
\end{proof}

Thanks to Lemma \ref{lemma v minimizer}, we can find an explicit formula for $v_{\AAA_R}$. In particular, we have the following theorem.

\begin{theorem}  \label{theorem v exact}
Let $v_{\AAA_R}$ be the minimizer of Problem \eqref{Pb 2}. Then,
\begin{equation}\label{v B_R exact}
v_{\AAA_R}(r,\omega) = \sum_{n\in\NN} v_n^R \eta_n(kr) e^{in\om} ,
\end{equation}
where
\begin{equation} \label{psi n def}
\eta_n(\rho) = Y_n(kR_0) J_n(\rho) - J_n(kR_0) Y_n(\rho),\quad \rho>0,
\end{equation}
and
\begin{equation}\label{v_nR}
v_n^R= - \frac{f_n \gamma_n^R}{c_n^R},
\end{equation}
with
\begin{equation}\label{c_nR}
c_n^R= \int_{R_0}^R \Big[ \rho k^2\eta_n'(k\rho)^2+ \Big(\rho k^2+\frac{n^2}{\rho}\Big) \eta_n(k\rho)^2 \Big] d\rho,
\end{equation}
and
\begin{multline}\label{gamma_nR}
\gamma_n^R = \frac{2}{\pi} ki(R-R_0) + \\
+ \frac{1}{\Hnu(kR_0)} \int_{R_0}^R [k^2 \rho \Hnu{}'(k\rho)\eta_n'(k\rho) + ( k^2\rho + \frac{n^2}{\rho} ) \Hnu(k\rho) \eta_n(k\rho) ] d\rho .
\end{multline}
\end{theorem}

\begin{proof}
Since $u_{\AAA_R}$ solves \eqref{Helm f}, then $v_{\AAA_R}=u_{\AAA_R} - \psi$ solves \eqref{Helm 0}. By separation of variables and from the homogeneous boundary condition on $\pa B_{R_0}$, we write a solution $v$ of \eqref{Helm 0} as
\begin{equation*}
v(r,\om) = \sum_{n\in\NN} v_n \eta_n(kr) e^{in\om},
\end{equation*}
where $\eta_n$ is given by \eqref{psi n def}. Since $V(r,\om)= e^{-ikr} v(r,\om)$, then
\begin{equation*}
V(r,\om)= \sum_{n\in\NN} v_n \teta_n(kr) e^{in\om},
\end{equation*}
where we set
\begin{equation*}
\teta_n(r)= e^{-ikr} \eta_n(kr).
\end{equation*}
By letting $\Psi(r,\om)= e^{-ikr} \psi(r,\om)$, we have that
\begin{equation*}
\Psi(r,\om)= \sum_{n\in\NN} \frac{f_n\tilde{h}_n(kr)}{\Hnu(kR_0)} e^{in\om},
\end{equation*}
where
\begin{equation*}
\tilde{h}_n(r) = e^{-ikr} \Hnu(kr).
\end{equation*}
We notice that
\begin{equation*}
\tilde{\eta}_n'(\rho)= ke^{-ik\rho} ( \eta_n(k\rho)-i\eta_n(k\rho) );
\end{equation*}
from Lemma \ref{lemma herm prod} and since $ \eta_n$ is real-valued, we have that
\begin{equation}\label{V scal V}
\langle V, V \rangle_R= \sum_{n\in\NN} |v_n|^2 c_n^R,
\end{equation}
where $c_n^R$ is given by \eqref{c_nR}. Analogously, from
\begin{equation*}
\langle \Psi, V \rangle_R= \RE \sum_{n\in\NN} \frac{f_n \bar{v}_n}{\Hnu(kR_0)} \bar{v}_n \int_{R_0}^R [\rho \tilde{h}_n'(\rho) \overline{ \teta_n'(\rho) } + \frac{n^2}{\rho} \tilde{h}_n(\rho) \overline{\teta_n(\rho)}]  d\rho,
\end{equation*}
we obtain that
\begin{multline}\label{A}
\rho \tilde{h}_n'(\rho) \overline{ \teta_n'(\rho) } + \frac{n^2}{\rho} \tilde{h}_n(\rho) \overline{\teta_n(\rho)} = \\ = \rho k^2(\Hnu{}'(k\rho)-i\Hnu(k\rho))(\eta_n'(k\rho)+ik\eta_n(k\rho)) + \frac{n^2}{\rho} \Hnu(k\rho)\eta_n(k\rho).
\end{multline}
Some computations yield
\begin{equation*}
\Hnu{}'(k\rho) \eta_n(k\rho) - \Hnu(k\rho)\eta_n'(k\rho) = \Hnu(kR_0)[J_n(k\rho)Y_n'(k\rho)-J_n'(k\rho)Y_n(k\rho)],
\end{equation*}
%
and, from
\begin{equation*}
J_n(r)Y_n'(r)-J_n'(r)Y_n(r)= \frac{2}{\pi r}
\end{equation*}
(see formula 9.1.16 in \cite{AS}), we obtain that
\begin{equation} \label{B}
\Hnu{}'(k\rho)\eta_n(k\rho) - \Hnu(k\rho) \eta_n'(k\rho) =
\frac{2}{\pi k \rho} \Hnu(kR_0).
\end{equation}
From \eqref{A} and \eqref{B} we have that
\begin{equation*}
\langle \Psi, V \rangle_R= \RE \sum_{n\in\NN}  f_n \gamma_n^R \bar{v}_n,
\end{equation*}
with $\gamma_n^R$ given by \eqref{gamma_nR} and hence
\begin{equation*}
\langle V+2\Psi,V \rangle_R = \sum_{n\in\NN} c_n^R|v_n|^2+ 2 \RE f_n \gamma_n^R \bar{v}_n.
\end{equation*}
By minimizing each term of the sum we obtain \eqref{v_nR}.
\end{proof}

In order to obtain estimates on the convergence, it will be useful to write the coefficients $c_n^R$ and $\gamma_n^R$ more explicitly. We will need the following lemma.

\begin{lemma} \label{lemma tecnico bessel}
Let $\Ccal_n$ and $\Dcal_n$ be two cylinder functions, with $n\in\NN$. Then
\begin{multline}\label{int bess nuovo}
\int^r  \Big[ \rho \Ccal_n'(\rho)  \Dcal_n'(\rho) + \Big( \rho + \frac{n^2}{\rho} \Big) \Ccal_n(\rho) \Dcal_n(\rho) \Big] d\rho =\\
\\ r^2 \Big( \Ccal_n(r) \Dcal_n(r) + \Ccal_n'(r) \Dcal_n'(r)  \Big)  + r \Ccal_n'(r) \Dcal_n(r) - n^2 \Ccal_n(r) \Dcal_n(r) \, .
\end{multline}
\end{lemma}

\begin{proof}
We multiply the Bessel equation
\begin{equation*}
\big(r \Ccal_n'(r) \big)' + \Big( r-\frac{n^2}{r} \Big) \Ccal_n(r) =0,
\end{equation*}
times $\Dcal(r)$ and integrate. After one integration by parts, we have that
\begin{equation*}
r \Ccal_n'(r) \Dcal_n(r) - \int^r \rho \Ccal_n'(\rho) \Dcal_n'(\rho) d\rho + \int^r \Big( \rho - \frac{n^2}{\rho} \Big) \Ccal_n(\rho) \Dcal_n(\rho) d\rho = 0,
\end{equation*}
and hence
\begin{equation*}
\int^r \rho \Ccal_n'(\rho) \Dcal_n'(\rho) d\rho + \int^r \Big( \rho + \frac{n^2}{\rho} \Big) \Ccal_n(\rho) \Dcal_n(\rho) d\rho = r \Ccal_n'(r) \Dcal_n(r) + 2 \int^r \rho \Ccal_n(\rho) \Dcal_n(\rho) d\rho .
\end{equation*}
From formula 10.22.5 in \cite{OLBC} we obtain
\begin{multline} \label{eq w 1}
\int^r \rho \Ccal_n'(\rho) \Dcal_n'(\rho) d\rho + \int^r \Big( \rho + \frac{n^2}{\rho} \Big) \Ccal_n(\rho) \Dcal_n(\rho) d\rho = \\ 
= r \Ccal_n'(r) \Dcal_n(r) + \frac{r^2}{2} \Big( 2 \Ccal_n(r) \Dcal_n(r) - \Ccal_{n-1}(r) \Dcal_{n+1}(r) -  \Ccal_{n+1}(r) \Dcal_{n-1}(r) \Big) .
\end{multline}
By using the recurrence relations 10.6.2 in \cite{OLBC},  we find that 
\begin{equation*}
\Ccal_{n-1}(r) \Dcal_{n+1}(r) +  \Ccal_{n+1}(r) \Dcal_{n-1}(r) = 2 \Ccal_n'(r) \Dcal_n'(r) + 2 \frac{n^2}{r^2} \Ccal_n(r)\Dcal_n(r),
\end{equation*}
and from \eqref{eq w 1} we conclude.
\end{proof}

Now we are ready to find estimates of the rate of convergence of the solution of the approximating problem to the exact solution.

\begin{theorem} \label{thm main}
Let $N\in\NN$ and $R_*>R_0$ be fixed. Let $\psi$ and $u_{\AAA_R}$ be the solutions of \eqref{helm f rad cond} and \eqref{Pb 1}, respectively, and assume that
\begin{equation*}
    f(\om)=\sum_{n=-N}^N f_n e^{i n \om},
\end{equation*}
with $\om \in [0,2\pi]$. Then
\begin{equation} \label{asympt 1}
\|\psi - u_{\AAA_R}\|_{H^1(\AAA_{R_*})} = \OO(R^{-1}),
\end{equation}
and
\begin{equation} \label{asympt 2}
\|\psi - u_{\AAA_R}\|_{H^1(\AAA_{R})} = \OO(R^{-1/2}),
\end{equation}
as $R\to + \infty$.
\end{theorem}

\begin{proof}
It will be enough to estimate the rate of convergence of the solution of $v_{\AAA_R}$ to zero, which clearly gives the desired $H^1$ estimate of the difference between the exact and the approximating solutions, as follows from \eqref{v minimiz def}.

Let $n \in \NN$ be fixed and let $c_n^R$ and $\gamma_n^R$ be given by \eqref{c_nR} and \eqref{gamma_nR}.
We use Theorem \ref{theorem v exact}, Lemma \ref{lemma tecnico bessel}, and the asymptotic formulas in Section 10.17 in \cite{OLBC} and we find the following asymptotic expansions for $R \to +\infty$:
\begin{equation*}
c_n^R = \frac{2kR}{\pi} \Big[ J_n(k R_0)^2  + Y_n(k R_0)^2 \Big]   + O(1),
\end{equation*}
and
\begin{multline*}
\frac{\pi}{2}\gamma_n^R = -ikR_0 -k^2  R_0^2 \frac{H_n^{(1)'} (kR_0)}{H_n^{(1)}(kR_0)} \eta_n'(kR_0) + \\
+ \frac{e^{i\chi_n}}{H_n^{(1)}(kR_0)} \Big[ \alpha_n H_n^{(2)}(kR_0) e^{i\chi_n} + i \Big(Y_n(kR_0)\cos \chi_n - J_n(kR_0) \sin \chi_n \Big) \Big] + \\ + O(R^{-1}), 
\end{multline*}
where 
$$
\alpha_n=\frac{(4n^2-1)(n^2-1)}{2},
$$
and 
$$\chi_n=kR - \Big(n + \frac 1 2\Big) \frac \pi 2 .$$
In particular we have that
\begin{equation} \label{ga su c}
\frac{\gamma_n^R}{c_n^R} = O(R^{-1}), \quad \textmd{as } \ R \to + \infty.
\end{equation}
Since $R_*$ is fixed, then there exists a constant $C$, not depending on $R$ such that 
$$
\|\eta_n\|_{H^1(\mathcal{A}_{R_*})} \leq C,
$$
for every $n $, and hence it is clear that \eqref{ga su c} implies \eqref{asympt 1}. To prove \eqref{asympt 2} we notice that formula 10.22.5 in \cite{OLBC} implies that
$$
\int_0^R (\eta_n(r)^2 + \eta_n'(r)^2) r \, dr = O(R), \quad \textmd{as } \ R \to +\infty,
$$
and from \eqref{ga su c} we conclude.
\end{proof}

\begin{remark}
In Theorem \ref{thm main} we assumed that the source $f$ can be expressed in terms of a finite sum of Fourier coefficients, which is the most interesting case for the numerical computations. For a general $f \in L^2$ it is not clear whether Theorem \ref{thm main} holds. Indeed, the error bounds that we used for the asymptotic expansions of cylindric functions may be not sufficient to guarantee the convergence (see Section 10.17(iii) in \cite{OLBC}) and a more refined argument is probably needed to estimate the rate of convergence of the approximating solution. 
\end{remark}

\bibliographystyle{amsalpha}

\end{document}